\documentclass[12pt,twoside]{amsart}
\usepackage{}
\usepackage{txfonts}
\usepackage{amsthm}
\usepackage{latexsym}
\usepackage{mathabx}
\usepackage[all]{xy}
\usepackage{amssymb}
\usepackage{amscd,mathrsfs,graphicx,mathrsfs}

\date{}
\allowdisplaybreaks[4] \footskip=15pt

\renewcommand{\uppercasenonmath}[1]{}

\numberwithin{equation}{section} \theoremstyle{plain}
\newtheorem*{thm*}{Main Theorem}
\newtheorem{thm}{Theorem}[section]
\newtheorem{cor}[thm]{Corollary}
\newtheorem*{cor*}{Corollary}
\newtheorem{lem}[thm]{Lemma}
\newtheorem*{lem*}{Lemma}

\newtheorem*{prop*}{Proposition}
\newtheorem{fact}[thm]{Fact}
\newtheorem*{fact*}{Fact}
\newtheorem*{que*}{Question}
\newtheorem{rem}[thm]{Remark}
\newtheorem*{rem*}{Remark}

\newtheorem*{ex*}{Example}

\newtheorem*{com*}{Construction}
\newtheorem{df}[thm]{Definition}
\newtheorem*{df*}{Definition}

\newtheorem*{ack*}{ACKNOWLEDGEMENTS}

\newcommand{\C}{\mbox{\rm C}}
\newcommand{\Ext}{\mbox{\rm Ext}}
\newcommand{\Hom}{\mbox{\rm Hom}}
\newcommand{\h}{\mbox{\rm H}}
\newcommand{\im}{\mbox{\rm Im}}
\newcommand{\G}{\mbox{\rm Gpd}}
\newcommand{\M}{\mbox{\rm Mod}}

\begin{document}
\begin{center}
{\Large \bf On triangle equivalences of stable categories}

\vspace{0.5cm}
\emph{In Memory of Professor Ragnar-Olaf Buchweitz}\\\

\footnotetext{* Corresponding author.

E-mail:~dizhenxing19841111@126.com (Z. Di),~liuzk@nwnu.edu.cn (Z. Liu), and~weijiaqun@njnu.edu.cn (J. Wei).}

 {\small Zhenxing Di$^{1}$, Zhongkui Liu$^{1}$ and Jiaqun Wei$^{1,2,*}$\\\

1.  Department of Mathematics, Northwest Normal University, \\ Lanzhou 730070, People's Republic of China \\

2.  {Institute of Mathematics, School of Mathematics Sciences, Nanjing Normal University, \\Nanjing 210023, People's Republic of China}\\
}\end{center}

\vskip.5cm \noindent{\bf Abstract.}
{We apply the Auslander-Buchweitz approximation theory to
show that the Iyama and Yoshino's subfactor triangulated category can be realized as a triangulated quotient.
Applications of this realization go in three directions.
Firstly, we recover both a result of Iyama and Yang and a result of the third author.
Secondly, we extend the classical Buchweitz's triangle equivalence
from Iwanaga-Gorenstein rings to Noetherian rings.
Finally, we obtain the converse of Buchweitz's triangle equivalence and a result of Beligiannis,
and give characterizations for Iwanaga-Gorenstein rings and Gorenstein algebras. }

\vskip.2cm

 \vskip.2cm


\section{\bf Introduction}

In the same conceptual framework of cluster tilting theory,
Iyama and Yoshino, working in the context of the theory of mutations,
proved in \cite{IYa2} that certain factors of extension closed subcategories of a triangulated category $\mathcal{T}$ are again triangulated.
More precisely, they showed that
if $(\mathcal{Z,Z})$ forms a $\mathcal{D}$-mutation pair in $\mathcal{T}$ (see \cite[Definition 2.5]{IYa2} or Definition \ref{Df of Mutation}),
where $\mathcal{Z}$ and $\mathcal{D}$ are subcategories of $\mathcal{T}$
such that $\mathcal{D}\subseteq\mathcal{Z}$, $\mathcal{Z}$ is extension-closed and $\Hom_{\mathcal{T}}(\mathcal{Z},\mathcal{D}[1]) = 0 = \Hom_{\mathcal{T}}(\mathcal{D},\mathcal{Z}[1])$,
then the subfactor category $\mathcal{Z}/[\mathcal{D}]$
possesses also the structure of a triangulated category
(see \cite[Section 4]{IYa2} for details).
This result played a key role in the development of mutation theory of tilting and silting objects (subcategories) in general triangulated categories
(see, for example, \cite{AI,IYa,NSZ,SO}).

Originated from the concept of injective envelopes,
the approximation theory has attracted increasing interest among scholars
and, hence, obtained the considerable development especially in the context of module categories since the fifties
(see, for example, \cite{AR,ASS,EJ2000}).
Inspired by the ideas of injective envelopes and projective covers,
Auslander and Buchweitz studied in \cite{AB} the maximal Cohen-Macaulay approximations for certain modules.
Indeed, they established their theory in the context of abelian categories,
and provided important applications in several settings.
Since the appearance of their work it has influenced numerous subsequent articles of researchers.
In particular, Mendoza Hern\'{a}ndez et al. developed in \cite{Msss,Msss1}
recently an analogous theory of approximations in the sense of Auslander and Buchweitz for triangulated categories.

The main purpose of this manuscript is to apply this Auslander-Buchweitz approximation theory
to give another description for the above Iyama and Yoshino's subfactor triangulated category.
Indeed, we obtain the following result,
which shows that under the condition that $\mathcal{D}$ is \emph{presilting} (that is, $\Hom_{\mathcal{T}}(\mathcal{D,D}[\geqslant1])=0$),
the subfactor triangulated category $\mathcal{Z}/[\mathcal{D}]$ can be realized as a triangulated quotient.
Here, denote by $\langle  \mathcal{D}\rangle$ (resp., $\langle  \mathcal{Z}\rangle$)
the smallest thick subcategory of $\mathcal{T}$ containing $\mathcal{D}$ (resp., $\mathcal{Z}$)
and by $\langle \mathcal{Z}\rangle/ \langle  \mathcal{D}\rangle$ the corresponding Verdier's triangulated quotient
(or simply, triangulated quotient) category.

\begin{thm}\label{triangulated quotient1}{\rm(=\,Theorem \ref{thm1})}
Let $\mathcal{D}\subseteq\mathcal{Z}$ be subcategories of $\mathcal{T}$
such that $(\mathcal{Z,Z})$ forms a $\mathcal{D}$-mutation pair.
Suppose that $\mathcal{D}$ is presilting.
Then there exists a triangle equivalence
$$\mathcal{Z}/[\mathcal{D}]\,\, \simeq \,\,\langle  \mathcal{Z}\rangle/ \langle  \mathcal{D}\rangle.$$
\end{thm}

As a special case of this realization,
we recover the main result \cite[Theorem A]{W1} of the third author (see Corollary \ref{wei's}).
Recently, Iyama and Yang showed in \cite[Theorem 3.6]{IYa} that under some conditions the so-called silting reduction of $\mathcal{T}$
can be realized as a certain subfactor triangulated category of $\mathcal{T}$.
While according to \cite[Corollary 2.7]{W1},
we see that \cite[Theorem 3.6]{IYa} can be deduced from \cite[Theorem A]{W1}.
Thus, \cite[Theorem 3.6]{IYa} is also a consequence of our result (see Corollary \ref{Yang1}).

In 1987, Buchweitz \cite{Buch} studied the triangulated quotient category
\begin{center}
$D_{sg}(R):= D^\mathrm{b}(\mathrm{mod}\,R)/ K^\mathrm{b}(\mathrm{proj}\,R)$,
\end{center}
where $D^\mathrm{b}(\mathrm{mod}\,R)$ is the bounded derived category of finitely generated modules over a Noetherian ring $R$
and $K^\mathrm{b}(\mathrm{proj}\,R)$ is the bounded homotopy category of finitely generated projective modules,
under the name of ``\,stable derived category''.
In particular, he established in \cite{Buch} the following famous triangle equivalence when $R$ is an Iwanaga-Gorenstein ring
$$\quad\quad\quad\underline{\mathcal{G}p}\,\,\simeq\,\,D_{sg}(R)\quad\quad\quad(\dagger),$$
where $\underline{\mathcal {G}p}$ denotes the stable category of the Frobenius category of all finitely generated Gorenstein projective modules.
In the representation theory of finite-dimensional algebras,
this triangulated quotient category appeared in Rickard's work \cite{Rick}.
It is proved therein that this category is triangle equivalent to the stable module category over a self-injective algebra.
Later, this result was generalized to Gorenstein Artin algebras via the (co)tilting theory by Happel \cite{Ha1}.
Recently, Orlov \cite{Orlov} reconsidered this triangulated quotient category
and called $D_{sg}(R)$ the \emph{singularity category} of the ring $R$
because this quotient category reflects certain homological singularity of the ring $R$.

As applications of Theorem \ref{triangulated quotient1},
we extend the above classical Buchweitz's triangle equivalence ($\dagger$)
from Iwanaga-Gorenstein rings to Noetherian rings (see Corollary \ref{Gproj} and Remark \ref{Buchweitz}),
and obtain the converse of Buchweitz's triangle equivalence ($\dagger$), see Corollary \ref{inverse1}.

Besides, other triangulated quotient categories have also attracted increasing interest among scholars
(we refer the reader to \cite{CHEN} for some basic knowledge about this topic).
Beligiannis studied the triangulated quotient category $D^\mathrm{b}(\M\,R)/ K^\mathrm{b}(\mathrm{Proj}\,R)$ for an arbitrary ring $R$,
where $D^\mathrm{b}(\M\,R)$ is the bounded derived category of modules
and $K^\mathrm{b}(\mathrm{Proj}\,R)$ is the bounded homotopy category of projective modules (see \cite{BE1}).
Just as the singularity category,
this quotient category reflects also the homological singularity of the ring $R$,
and it treats modules which are not necessarily finitely generated.
It seems like to refer such a triangulated quotient category as the \emph{big singularity category} of the ring $R$
(note that over a Noetherian ring $R$ the homotopy category of all acyclic complexes of injective modules $K_{\textrm{ac}}(\textrm{Inj}\,R)$ is
compactly generated by $D_{sg}(R)$ (see \cite[Proposition 2.3]{HK} or \cite[Theorem 2.20]{SO1}).
The category $K_{\textrm{ac}}(\textrm{Inj}\,R)$ is also called the big singularity category in \cite{SO1}.
We would like to remind readers of the difference between the two settings).
In particular, Beligiannis showed in \cite[Theorem 6.9]{BE1} that $R$ has finite Gorenstein global dimension
if and only if there exists the triangle equivalence
\begin{center}
$\quad\quad\quad\underline{\mathcal{G}P}\,\,\cong\,\, D^\mathrm{b}(\M\,R)/ K^\mathrm{b}(\mathrm{Proj}\,R),$
\end{center}
where $\underline{\mathcal {G}P}$ denotes the stable category of the Frobenius category of all Gorenstein projective modules.
This result extends Buchweitz's triangle equivalence ($\dagger$) to the `big' version,
and provides the corresponding converse.

As another application of Theorem \ref{triangulated quotient1},
we also obtain the above Beligiannis' result (see Corollary \ref{inverse2}).

Recently, Bergh, J{\o}rgensen and Oppermann proved in \cite[Theorem 3.6]{BJS}
that if $R$ is either a left and right Artin ring or a commutative Noetherian local ring,
then there exists a triangle equivalence
$\underline{\mathcal {G}p} \simeq D^\mathrm{b}(\mathrm{mod}\,R)/ K^\mathrm{b}(\mathrm{proj}\,R)$
if and only if $R$ is Gorenstein.
As a further result,
we obtain the following result characterizing Iwanaga-Gorenstein rings and Gorenstein algebras.
According to their result,
we note that the equivalence of (1) and (3) is indeed a special case of \cite[Theorem 3.6]{BJS}.

\begin{cor}\label{m2}{\rm(=\,Corollary \ref{commutative N})}
Let $R$ be a left and right Noetherian ring.
Then the following statements are equivalent:

$(1)$ $R$ is Iwanaga-Gorenstein.

$(2)$ There exists a triangle equivalence
      $$\underline{\mathcal {G}P} \simeq D^\mathrm{b}(\M\,R)/ K^\mathrm{b}(\mathrm{Proj}\,R).$$
If $R$ is further an Artin algebra,
then the above two conditions are equivalent to

$(3)$  There exists a triangle equivalence
      $$\underline{\mathcal {G}p} \simeq D^\mathrm{b}(\mathrm{mod}\,R)/ K^\mathrm{b}(\mathrm{proj}\,R).$$
\end{cor}

We conclude this section by summarizing the contents of this article.
Section 2 contains necessary notions and results for use throughout this article.
In Section 3, we present the proof of Theorem \ref{triangulated quotient1}.
As applications of Theorem \ref{triangulated quotient1},
some triangle equivalences on stable categories
will be displayed in Section 4,
where we give also the proof of Corollary \ref{m2}.

\section{\bf Preliminaries}

In this section, we fix some notation.
We recall the Auslander-Buchweitz approximation triangles in a triangulated category,
and the triangle structure of the Iyama and Yoshino's subfactor triangulated category associated to a mutation pair.
We recall definitions of a presilting subcategory and a co-t-structure,
and give some necessary facts about these notions.

\subsection{Some notation}

\emph{Throughout this article},
by the term \emph{``subcategory''} we always mean a full additive subcategory of an additive category
closed under isomorphisms and direct summands.

Let $\mathcal{A}$ be an additive category.
For an ideal $\mathcal{I}$,
denote by $\mathcal{A}/\mathcal{I}$ the category
whose objects are objects of $\mathcal{A}$ and
whose morphisms are elements of
\begin{center}
$\Hom_{\mathcal{A}}(M,N)/\mathcal{I}(M,N)\quad$ for all $M,N\in \mathcal{A}/\mathcal{I}$.
\end{center}
Suppose that $\mathcal{D}$ is a subcategory of $\mathcal{A}$.
Denote by $[\mathcal{D}]$ the ideal of $\mathcal{A}$ consisting of all morphisms factoring through some object in $\mathcal{D}$.
Thus, we have a category $\mathcal{A}/[\mathcal{D}]$,
which is also an additive category.

\emph{Throughout this article}, let $R$ denote an associative ring with identity.
Denote by $\M\,R$ the category of all right $R$-modules,
by mod$\,R$ the category of all finitely generated right $R$-modules,
by Proj$\,R$ the category of all projective right $R$-modules,
and by proj$\,R$ the category of all finitely generated projective right $R$-modules.

\emph{Throughout this article},
let $\mathcal{T}$ be a triangulated category.
We will denote by [1] the shift functor of any triangulated category unless otherwise stated.
Suppose that $\mathcal{C}$ is a subcategory of $\mathcal{T}$.
Denote by $\langle C\rangle$ the \emph{smallest thick subcategory} of $\mathcal{T}$ containing $\mathcal{C}$.
For some integer $n$, set
\begin{center}
$\,\,\,\quad\mathcal{C}^{\perp_{i>n}}=\{N\in\mathcal{T}\,|\,\Hom_{\mathcal{T}}(M,N[>n]) = 0~{\rm for~all}~M \in \mathcal{C} \},$

$\,\,\,\quad ^{\perp_{i>n}}\mathcal{C}\,=\{N\in\mathcal{T}\,|\,\Hom_{\mathcal{T}}(N, M[>n]) = 0~{\rm for~all}~M \in \mathcal{C} \}.$
\end{center}

Following the notions in \cite{W2},
the subcategory $\mathcal{C}$ is called \emph{extension-closed}
if for any triangle
$$U\to V \to W \to U[1]$$ in $\mathcal{T}$ with $U,W \in \mathcal{C}$, it holds that $V \in \mathcal{C}$.
It is \emph{resolving} (resp., \emph{coresolving})
if it is further closed under the functor $[-1]$ (resp., $[1]$).
Note that $\mathcal{C}$ is resolving (resp., coresolving) if and only if,
for any triangle
\begin{center}
$U \to V \to W \to U[1]\quad\quad$ (resp., $W \to V \to U \to W[1]$)
\end{center}
in $\mathcal{T}$ with $W \in \mathcal{C}$,
it holds that $U \in \mathcal{C} \Leftrightarrow V \in \mathcal{C}$.
It is easy to see that $\mathcal{C}^{\perp_{i>0}}$ (resp., $^{\perp_{i>0}}\mathcal{C}$) is coresolving (resp., resolving).

Let now $\mathcal{A}$ be an abelian category.
A complex $X$ is often displayed as a sequence
$$\xymatrix@C=0.5cm{
\cdots \ar[r] & X_{n-1} \ar[rr]^{\delta^X_{n-1}} && X_n \ar[rr]^{\delta^X_n} && X_{n+1} \ar[r]^{} & \cdots}$$
of objects in $\mathcal{A}$ with $\delta^X_{n}\delta^X_{n-1}$ for all $n\in \mathbb{Z}$.
The $n$th \emph{homology} of $X$ is defined as Ker$\delta^X_n/\im\delta^X_{n-1}$
and denoted by H$_n(X)$.
Set $\C_n(X)=\mathrm{Coker}\delta^X_{n-1}$.
We say that two complexes $X$ and $Y$ are \emph{equivalent}, and denoted by $X\simeq Y$ \cite[A.1.11, p. 164]{lwc1},
if they can be linked by a sequence of quasi-isomorphisms with arrows in alternating directions.

Let $\mathcal {E}$ be a subcategory of $\mathcal{A}$.
Denote by $D^{\rm b}(\mathcal{A})$ the bounded derived category of $\mathcal{A}$,
by $D^{-}(\mathcal{A})$ the derived category of bounded-above complexes,
and by $K^{\rm b}(\mathcal{E})$ the bounded homotopy category with each complex constructed by objects in $\mathcal{E}$.

\subsection{Auslander-Buchweitz approximation triangles}

We recall in this subsection the Auslander-Buchweitz
approximation triangles established by Mendoza Hern\'{a}ndez et al. in \cite{Msss}.

Let $\mathcal{W}$ and $\mathcal{X}$ be subcategories of $\mathcal{T}$.
For a non-negative integer $n$,
denote by $(\widehat{\mathcal{X}})_n$ (resp., $(\widecheck{\mathcal{X}})_n$) the class consisting of all objects $T$
satisfying that there exists a series of triangles
\begin{center}
$T_{i+1}\to X_i \to T_i \to T_{i+1}[1]\quad\quad$ (resp., $T_i\to X_i \to T_{i+1} \to T_i[1]$)
\end{center}
in $\mathcal{T}$ with $0\leqslant i \leqslant n$ such that $T_0 =T$, $T_{n+1}=0$ and each $X_i\in \mathcal{X}$.
We use the symbol $\widehat{\mathcal{X}}$ (resp., $\widecheck{\mathcal{X}}$) to stand for
the class consisting of all objects $K$
satisfying that there is a non-negative integer $m$ such that $K\in(\widehat{\mathcal{X}})_m$ (resp., $K\in(\widecheck{\mathcal{X}})_m$).
Note that $0\in\mathcal{X}$ by assumption.
It is easy to see that $\widehat{\mathcal{X}}$ (resp., $\widecheck{\mathcal{X}}$) is closed under the functor $[1]$ (resp., $[-1]$).

Recall that $\mathcal{W}$ is called a \emph{weak-cogenerator} in $\mathcal{X}$ \cite[Definition 5.1]{Msss}
if $\mathcal{W}\subseteq\mathcal{X}$ and for any object $X\in \mathcal{X}$,
there exists a triangle $$X\to W \to X' \to X[1]$$ in $\mathcal{T}$ with $X'\in\mathcal{X}$ and $W\in\mathcal{W}$.
The subcategory $\mathcal{W}$ is said to be $\mathcal{X}$-\emph{injective} if $\Hom_{\mathcal{T}}(X,W[\geqslant1])=0$
for any object $W\in\mathcal{W}$ and any object $X\in\mathcal{X}$.
Dually, one have the notions of $\mathcal{W}$ being a \emph{weak-generator} in $\mathcal{X}$ and $\mathcal{X}$-\emph{projective}.
We say that $\mathcal {W}$ is a \emph{weak-generator-cogenerator} in $\mathcal {X}$
if it is both an $\mathcal{X}$-projective weak-generator and an $\mathcal{X}$-injective weak-cogenerator in $\mathcal {X}$.

The following two results will be used frequently in the sequel.

\begin{thm} $($\cite[Theorem 5.4]{Msss}$)$\label{A-B Approximation1}
Let $\mathcal{W} \subseteq \mathcal{X}$ be subcategories of $\mathcal{T}$.
Suppose that $\mathcal{X}$ is closed under extensions and $\mathcal{W}$ is a weak-cogenerator in $\mathcal{X}$.
Then for any object $M \in \widehat{\mathcal{X}}$,
there exist triangles
\begin{center}
$K_M \to X_M \to M\to K_M[1]\quad$ and $\quad M \to K^M \to X^M\to M[1]$
\end{center}
in $\mathcal{T}$ with $X_M, X^M\in \mathcal{X}$ and $K_M,K^M\in \widehat{\mathcal{W}}$.
\end{thm}

Dually, one has

\begin{thm}\label{A-B Approximation2}
Let $\mathcal{V} \subseteq \mathcal{Y}$ be subcategories of $\mathcal{T}$.
Suppose that $\mathcal{Y}$ is closed under extensions and $\mathcal{V}$ is a weak-generator in $\mathcal{Y}$.
Then for any object $N \in \widecheck{\mathcal{Y}}$, there exist triangles
\begin{center}
$N \to Y^N \to L^N \to N[1]\quad$ and $\quad Y_N\to L_N\to N\to Y_N[1]$
\end{center}
in $\mathcal{T}$ with $Y^N,Y_N \in \mathcal{Y}$ and $L^N,L_N\in \widecheck{\mathcal{V}}$.
\end{thm}

\subsection{Presilting and thick subcategories}\label{semi-thick}

In this subsection,
we mainly recall the definition of a presilting subcategory and give
some necessary facts on subcategories arising from a presilting subcategory.

\begin{df}\label{df of silting}$($\cite[Definition 2.1]{AI}$)$
{\rm Let $\mathcal{M}$ be a subcategory of $\mathcal{T}$.
Then $\mathcal{M}$ is called \emph{presilting} if
$\Hom_{\mathcal{T}}(M,M'[\geqslant1]) = 0$ for all objects $M, M'\in \mathcal{M}$.}
\end{df}

Let $\mathcal{M}$ be a presilting subcategory of $\mathcal{T}$.
We use the symbol $_{\mathcal {M}}\mathscr{X}$ (resp., $\mathscr{X}_{\mathcal {M}}$) to denote
the subcategory of $\mathcal {M}^{\perp_{i>0}}$ (resp., $^{\perp_{i>0}}\mathcal {M}$)
consisting of all objects $N$ such that there exist triangles
\begin{center}
$ N_{i+1}\to M_i\to N_i\to N_{i+1}[1]$
\quad
(resp., $N_{i}\to M_i\to N_{i+1}\to N_i[1]$)
\end{center}
in $\mathcal{T}$ such that $N_0=N$, $N_i\in\mathcal {M}^{\perp_{i>0}}$
(resp., $N_i\in {^{\perp_{i>0}}\mathcal {M}}$) and $M_i\in\mathcal {M}$ for all $i\geqslant0$.
It is easy to see that $\widehat{\mathcal {M}}\subseteq\, _{\mathcal{M}}\mathscr{X} \subseteq \mathcal {M}^{\perp_{i>0}}$
and $\widecheck{\mathcal {M}}\subseteq\,\mathscr{X}_{\mathcal {M}} \subseteq{ ^{\perp_{i>0}}\mathcal {M}}$.

\begin{lem}\label{necessary facts}$($\cite[Lemma 2.2]{W1}$)$
Let $\mathcal{M}$ be a presilting subcategory of $\mathcal{T}$.
Then the following statements hold:

$(1)$ $\widecheck{\mathcal {M}}$ and $\mathscr{X}_{\mathcal {M}}$ are resolving.

$(2)$ $\widehat{\mathcal {M}}$ and $_{\mathcal {M}}\mathscr{X}$ are coresolving.
\end{lem}

In the rest of this subsection,
we consider thick subcategories of $\mathcal{T}$.
Let $\mathcal{H}$ be a subcategory of $\mathcal{T}$.
Define
\begin{center}
$({\mathcal{H}})_{+}:=\{ N\in\mathcal{T } \mid N\cong L[i]$ for some object $L\in {\mathcal{H}}$ and some integer $i\geqslant 0\}$.
\end{center}
\begin{center}
$({\mathcal{H}})_{-}:=\{ N\in\mathcal{T } \mid N\cong L[i]$ for some object $L\in {\mathcal{H}}$ and some integer $i\leqslant 0\}$.
\end{center}

\begin{lem}\label{thick}\cite[Lemma 2.1]{W1}
Let $\mathcal{H}$ be a subcategory of $\mathcal{T}$.

$(1)$ If $\mathcal{H}$ is resolving,
then $\langle\mathcal{H}\rangle=({\mathcal{H}})_{+}=\widehat{\mathcal{H}}$.

$(2)$ If $\mathcal{H}$ is coresolving,
then $\langle\mathcal{H}\rangle=({\mathcal{H}})_{-}=\widecheck{\mathcal{H}}$.
\end{lem}

\begin{cor}\label{weak-generator-cogenerator}
Let $\mathcal{C}$ be a subcategory of $\mathcal{T}$ closed under extensions.
If $\mathcal{C}$ admits a weak-generator-cogenerator,
then $\langle\mathcal{C}\rangle=(\widecheck{\mathcal{C}})_{+}=(\widehat{\mathcal{C}})_{-}$.
\end{cor}

\begin{proof}
Note that $\mathcal{C}$ admits a weak-generator-cogenerator  by assumption.
It is easy to see that both $\widecheck{\mathcal{C}}$ and $\widehat{\mathcal{C}}$ are closed under extensions and direct summands.
Hence, $\widecheck{\mathcal{C}}$ (resp., $\widehat{\mathcal{C}}$) is resolving (resp., coresolving),
and so $(\widecheck{\mathcal{C}})_{+}=\langle\widecheck{\mathcal{C}}\rangle$ and
$(\widehat{\mathcal{C}})_{+}=\langle\widehat{\mathcal{C}}\rangle$ by Lemma \ref{thick}.
However, it is clear that $\langle\widecheck{\mathcal{C}}\rangle=\langle\mathcal{C}\rangle$
(resp., $\langle\widehat{\mathcal{C}}\rangle=\langle\mathcal{C}\rangle$).
Therefore, the result follows.
\end{proof}

Suppose that $\mathcal{M}$ is a presilting subcategory of $\mathcal{T}$.
According to \cite[Lemma 5.3(2)]{Msss1},
we know that $\mathcal{M}$ is closed under extensions.
Moreover, it is obvious that $\mathcal{M}$ admits itself as a weak-generator-cogenerator.
Thus, by Corollary \ref{weak-generator-cogenerator}, we obtain

\begin{cor}\label{lemm+}
Let $\mathcal{M}$ be a presilting subcategory of $\mathcal{T}$.
Then $\langle\mathcal{M}\rangle=(\widecheck{\mathcal{M}})_{+}=(\widehat{\mathcal{M}})_{-}$.
\end{cor}

\subsection{Mutation pair, subfactor triangulated category and co-t-structure}

We recall in this subsection the triangle structure of the Iyama and Yoshino's subfactor triangulated category associated to a mutation pair,
and the definition of a co-t-structure.

\begin{df}\label{Mutation pair}$($\cite[Sections 4]{IYa2}$)$\label{Df of Mutation}
{\rm Let $\mathcal{D}$ and $\mathcal{Z}$ be subcategories of $\mathcal{T}$
such that $\mathcal{D}\subseteq\mathcal{Z}$.
The pair $(\mathcal{Z},\mathcal{Z})$ is called a $\mathcal{D}$-\emph{mutation pair}
if the following conditions hold:

(1) $\mathcal{Z}$ is closed under extensions in $\mathcal{T}$.

(2) For any object $Z \in \mathcal{Z}$,
there exist triangles
\begin{center}
$Z^+ \to D^+ \to Z \to Z^+[1]\quad$ and $\quad Z\to D^- \to Z^- \to Z[1]$
\end{center}
in $\mathcal{T}$ with $D^{\pm} \in \mathcal{D}$ and $ Z^{\pm}\in \mathcal{Z}$.

(3) $\Hom_{\mathcal{T}}(\mathcal{Z},\mathcal{D}[1]) = 0 = \Hom_{\mathcal{T}}(\mathcal{D},\mathcal{Z}[1])$.}
\end{df}

\begin{rem}\label{properties of presilting}{\rm
If $\mathcal{D}$ is further a presilting subcategory of $\mathcal{T}$ in the above definition,
then it is not hard to check that $$\Hom_{\mathcal{T}}(\mathcal{Z},\mathcal{D}[\geqslant1]) = 0 = \Hom_{\mathcal{T}}(\mathcal{D},\mathcal{Z}[\geqslant1]).$$
Therefore, the condition (2) implies that $\mathcal{D}$ is indeed a weak-generator-cogenerator in $\mathcal{Z}$ in this case.}
\end{rem}

\begin{thm}\label{triangulated structure}$($\cite[Definition 4.1 and Theorem 4.2]{IYa2}$)$
Let $\mathcal{D}\subseteq \mathcal{Z}$ be subcategories of $\mathcal{T}$
such that $(\mathcal{Z,Z})$ forms a $\mathcal{D}$-mutation pair.
Then the additive quotient category $\mathcal{Z}/[\mathcal{D}]$ has the structure of a triangulated category
with respect to the following shift functor and triangles:

$(1)$ For an object $Z \in \mathcal{Z}$,
take a fixed triangle
$$Z \overset{}{\rightarrow} D_Z \overset{}{\rightarrow} Z\langle1\rangle \overset{}{\rightarrow}Z[1] $$
in $\mathcal{T}$ with $D_Z\in\mathcal{D}$ and $Z\langle1\rangle\in \mathcal{Z}$ $($see the condition (2) of Definition \ref{Mutation pair}$)$.
Then $\langle1\rangle$ gives a well-defined auto-equivalence
of $\mathcal{Z}/[\mathcal{D}]$, which is the shift functor of $\mathcal{Z}/[\mathcal{D}]$.

$(2)$ Suppose that
$$Z\overset{f}{\longrightarrow} Z'\overset{g}{\longrightarrow} Z'' \overset{h}{\longrightarrow}Z[1]$$
is a triangle in $\mathcal{T}$ with $Z, Z', Z'' \in \mathcal{Z}$.
Consider the following commutative diagram of triangles:
$$\xymatrix{Z \ar[r]^{f} \ar@{=}[d]& Z' \ar[r]^{g}\ar[d] & Z'' \ar[r]^{h}\ar[d]^{\delta} & Z[1]\ar@{=}[d] \\
Z \ar[r]^{f'} & D_Z \ar[r]^{g'}& Z\langle1\rangle \ar[r]^{h'}& Z[1]}$$
Then we have a complex $Z\overset{[f]}{\longrightarrow} Z'\overset{[g]}{\longrightarrow} Z'' \overset{[\delta]}{\longrightarrow}Z\langle1\rangle$ in $\mathcal{Z}/[\mathcal{D}]$.
The triangles in $\mathcal{Z}/[\mathcal{D}]$ are defined as the complexes which are isomorphic to a complex obtained in this way.
\end{thm}

\begin{df}\label{co-t-structure}\cite{BON,pau}\,\,{\rm
A \emph{co-t-structure} on $\mathcal{T}$ is a pair $(\mathcal{A,B})$ of subcategories of $\mathcal{T}$ such that

(1) $\mathcal{A}[-1]\subseteq \mathcal{A}$ and $\mathcal{B}[1]\subseteq \mathcal{B}$,

(2) Hom${_\mathcal{T} (\mathcal{A}[-1],\mathcal{B}) = 0}$, and

(3) $\mathcal{T} = \mathcal{A}[-1]\ast \mathcal{B}$.}
\end{df}

\begin{fact}\label{presilting CO-T}
Let $\mathcal{D}$ be a presilting subcategory of $\mathcal{T}$.
According to \cite[Theorem 5.5]{Msss1},
we know that the pair $({_{\mathcal{D}}\mathcal{U}},\mathcal{U}_{\mathcal{D}})$ forms a co-t-structure on $\langle\mathcal{D}\rangle$.
Here, the symbol ${_{\mathcal{D}}\mathcal{U}}$ (resp.,$\mathcal{U}_{\mathcal{D}}$) stands for
the smallest extension-closed subcategory of $\mathcal{T}$ containing $\mathcal{D}[\leqslant 0]$ (resp.,$\mathcal{D}[\geqslant 1]$).
\end{fact}

\section{\bf Realise the subfactor triangulated category as a triangulated quotient}

Throughout this section, let $\mathcal{D}\subseteq\mathcal{Z}$ be two subcategories of $\mathcal{T}$
such that $(\mathcal{Z,Z})$ forms a $\mathcal{D}$-mutation pair.
We show in this section that under the condition that $\mathcal{D}$ is \emph{presilting},
the subfactor triangulated category $\mathcal{Z}/[\mathcal{D}]$ is triangle equivalent to
the triangulated quotient $\langle\mathcal{Z}\rangle/ \langle  \mathcal{D}\rangle$ (see Theorem \ref{thm1}).

We begin with the following result,
which will be used in the proof of Lemma \ref{key for faithful}.

\begin{lem}\label{keyl}
Suppose that $\mathcal{D}$ is presilting.
Then we have

$(1)$ $\widehat{\mathcal{D}}=\mathcal{D}^{\perp_{i>0}}\cap\langle \mathcal{D}\rangle$ and

$(2)$ $\mathcal{D}= \widehat{\mathcal{D}}\cap{^{\perp_{i>0}}\mathcal{D}}$.
\end{lem}

\begin{proof}
(1) We only need to show that the containment
$\mathcal{D}^{\perp_{i>0}}\cap\langle \mathcal{D}\rangle\subseteq\widehat{\mathcal{D}}$
holds true.
To this end, let $N$ be an object in $\mathcal{D}^{\perp_{i>0}}\cap\langle \mathcal{D}\rangle$.
Since $N\in \langle \mathcal{D}\rangle$,
we see that $N\in (\widehat{\mathcal{D}})_{-}$ by Corollary \ref{lemm+}.
This implies that there exist some integer $i\leqslant0$ and an object $L\in \widehat{\mathcal{D}}$
such that $N\cong L[i]$.
If $i=0$ then $N\in \widehat{\mathcal{D}}$, as desired.
Suppose now that $i<0$.
Then by the definition of objects being in $\widehat{\mathcal{D}}$,
there exists a triangle $N[-i-1]\to N_1 \to D\overset f\longrightarrow N[-i]$
in $\mathcal{T}$ with $D\in \mathcal{D}$ and $N_1\in\widehat{\mathcal{D}}$.
Note that $N\in \mathcal{D}^{\perp_{i>0}}$ as well.
We see that $f=0$.
Therefore, $N_1\cong D\oplus N[-i-1]$.
According to \cite[Lemma 2.2(2)]{W1},
we know that $\widehat{\mathcal{D}}$ is closed under direct summands.
Hence, $N[-i-1]\in\widehat{\mathcal{D}}$.
Continuing the process,
we can finally obtain that $N\in\widehat{\mathcal{D}}$,
as desired.
Thus, we have
$\mathcal{D}^{\perp_{i>0}}\cap\langle \mathcal{D}\rangle\subseteq\widehat{\mathcal{D}}$.

(2) It suffices to show that the containment
$\widehat{\mathcal{D}}\cap{^{\perp_{i>0}}\mathcal{D}}\subseteq\mathcal{D}$ holds true.
Suppose that $K$ is an object in $\widehat{\mathcal{D}}\cap{^{\perp_{i>0}}\mathcal{D}}$.
Then there exist an integer $t\geqslant0$ and a series of triangles
\begin{center}
$K_{i+1}\to D_i \to K_i \to K_{i+1}[1]$
\end{center}
in $\mathcal{T}$ with $D_i\in \mathcal{D}$ for all $0\leqslant i \leqslant t$, $K_0=K$ and $K_{t+1}=0$.
Since both $K$ and $D_0$ belong to ${^{\perp_{i>0}}\mathcal{D}}$,
we deduce that $K_1\in{^{\perp_{i>0}}\mathcal{D}}$.
Consequently, each $K_i\in{^{\perp_{i>0}}\mathcal{D}}$.
Note that $K_t\cong D_t\in\mathcal{D}$.
It follows that the triangle
$K_t\to D_{t-1}\to K_{t-1}\to K_t[1]$
is split.
This implies that $K_{t-1}\in \mathcal{D}$.
Repeating the process,
we can finally obtain $K=K_0\in\mathcal{D}$,
as desired.
Thus, we have
$\widehat{\mathcal{D}}\cap{^{\perp_{i>0}}\mathcal{D}}\subseteq\mathcal{D}$.
 \end{proof}

The coming result will be applied in the proof of Theorem \ref{thm1}.

\begin{lem}\label{key for faithful}
Suppose that $\mathcal{D}$ is presilting.
Then $\mathcal{Z}\cap\langle \mathcal{D}\rangle=\mathcal{D}$.
\end{lem}

\begin{proof}
It suffices to show that the containment $\mathcal{Z}\cap\langle \mathcal{D}\rangle\subseteq\mathcal{D}$ holds true.
To this end, let $K$ be an object in $\mathcal{Z}\cap\langle \mathcal{D}\rangle$.
Note that $K\in\mathcal{Z}\subseteq \mathcal{D}^{\perp_{i>0}}$ (see Remark \ref{properties of presilting}).
By Lemma \ref{keyl}(1),
we see that $K\in \widehat{\mathcal{D}}$.
Since $K\in\mathcal{Z}\subseteq {^{\perp_{i>0}}\mathcal{D}}$ as well (see Remark \ref{properties of presilting} again),
we conclude that $K\in\mathcal{D}$ by Lemma \ref{keyl}(2).
Hence, $\mathcal{Z}\cap\langle \mathcal{D}\rangle\subseteq\mathcal{D}$, as desired.
\end{proof}

Let $F$ be the composition of functors:
$$\mathcal {Z}\hookrightarrow \langle\mathcal{Z}\rangle \rightarrow   \langle\mathcal{Z}\rangle/ \langle  \mathcal{D}\rangle$$
in which the latter one is the natural quotient functor.
It is clear that $F$ sends any object in $\mathcal {D}$ to zero in $ \langle\mathcal{Z}\rangle/ \langle  \mathcal{D}\rangle$,
so it factors through the subfactor triangulated category $\mathcal{Z}/[\mathcal{D}]$.
Consequently, there exists a functor
$$\overline{F}:\,\,\mathcal{Z}/[\mathcal{D}]\to\langle\mathcal{Z}\rangle/ \langle  \mathcal{D}\rangle$$
such that $F=\overline{F}\pi$,
where $\pi:\mathcal {Z}\to {\mathcal{Z}/[\mathcal{D}]}$
is the natural quotient functor.

Now, we are in a position to give the main result of the article.

\begin{thm}\label{thm1}
Suppose that $\mathcal{D}$ is presilting.
Then the functor
$$\overline{F}:\,\,\mathcal{Z}/[\mathcal{D}]\to\langle\mathcal{Z}\rangle/ \langle  \mathcal{D}\rangle$$
is a triangle equivalence.
\end{thm}

\begin{proof}

We need to show that $\overline{F}$ is a triangle functor,
and it is essentially surjective (or dense), full and faithful.

(1) $\overline{F}$ is a triangle functor.

Let $Z$ be an object in $\mathcal{Z}$.
Send the triangle
$$Z \overset{}{\rightarrow} D_Z \overset{}{\rightarrow} Z\langle1\rangle \overset{}{\rightarrow}Z[1]$$
in $\langle\mathcal{Z}\rangle$ (see Theorem \ref{triangulated structure}(1)) to $\langle\mathcal{Z}\rangle/ \langle  \mathcal{D}\rangle$.
Since both $D_Z$ and $D_Z[1]$ become zero in $\langle\mathcal{Z}\rangle/ \langle  \mathcal{D}\rangle$,
we obtain an isomorphism $ Z\langle1\rangle \overset{}{\rightarrow}Z[1]$ in $\langle\mathcal{Z}\rangle/ \langle  \mathcal{D}\rangle$.
This yields a natural isomorphism $\overline{F}\circ\langle1\rangle\cong[1]\circ\overline{F}$.

On the other hand, let
$$Z\overset{}{\rightarrow}Z'\rightarrow Z'' \rightarrow Z\langle1\rangle$$
be a triangle in $\mathcal{Z}/[\mathcal{D}]$.
Then we may assume that it comes from a commutative diagram
$$\xymatrix{Z \ar[r]^{} \ar@{=}[d]& Z' \ar[r]^{}\ar[d] & Z'' \ar[r]^{}\ar[d]^{} & Z[1]\ar@{=}[d] \\
Z \ar[r]^{} & D_Z \ar[r]^{}& Z\langle1\rangle \ar[r]^{}& Z[1]}$$
of triangles in $\langle\mathcal{Z}\rangle$ (see Theorem \ref{triangulated structure}(2)).
Applying the quotient functor $\langle\mathcal{Z}\rangle \rightarrow   \langle\mathcal{Z}\rangle/ \langle  \mathcal{D}\rangle$ to the above diagram
we have $Z\langle1\rangle\cong Z[1]$ in $\langle\mathcal{Z}\rangle/ \langle  \mathcal{D}\rangle$.
Thus,
$$Z\overset{}{\rightarrow}Z'\rightarrow Z'' \rightarrow Z[1]$$
is a triangle in $\langle\mathcal{Z}\rangle/ \langle  \mathcal{D}\rangle$.
It follows that $\overline{F}$ is a triangle functor.

(2) $\overline{F}$ is essentially surjective (or dense).

Let $M$ be an object in $\langle\mathcal{Z}\rangle/ \langle  \mathcal{D}\rangle$.
According to Remark \ref{properties of presilting},
we see that $\mathcal{D}$ is a weak-generator-cogenerator in $\mathcal{Z}$.
It follows from Corollary \ref{weak-generator-cogenerator} that $\langle\mathcal{Z}\rangle=(\widecheck{\mathcal{Z}})_{+}$.
Hence, $M\cong L[i]$ for some object $L\in\widecheck{\mathcal{Z}}$ and some integer $i\geqslant0$ by definition.

If $i=0$ then $M\in\widecheck{\mathcal{Z}}$.
Therefore, by Theorem \ref{A-B Approximation2},
there exists a triangle
$$M\to Z_1 \to K_1\to M[1]$$
in $\langle\mathcal{Z}\rangle$
with $Z_1\in \mathcal{Z}$ and $K_1\in \widecheck{\mathcal{D}}\subseteq\langle\mathcal{D}\rangle$.
Applying the quotient functor $\langle\mathcal{Z}\rangle \rightarrow   \langle\mathcal{Z}\rangle/ \langle  \mathcal{D}\rangle$ to the triangle,
we have $M\cong Z_1$ in $\langle\mathcal{Z}\rangle/ \langle  \mathcal{D}\rangle$, as desired.

Assume now that $i>0$. Then $M[-i]\in\widecheck{\mathcal{Z}}$.
By Theorem \ref{A-B Approximation2} again,
we obtain a triangle
$$M[-i]\to Z_2 \to K_2\to M[-i+1]$$
in $\langle\mathcal{Z}\rangle$
with $Z_2\in \mathcal{Z}$ and $K_2\in \widecheck{\mathcal{D}}\subseteq\langle\mathcal{D}\rangle$.
Hence, $M\cong Z_2[i]$ in $\langle\mathcal{Z}\rangle/ \langle  \mathcal{D}\rangle$.
We show next that $Z_2[i]\cong Z_2\langle i\rangle$ in $\langle\mathcal{Z}\rangle/ \langle  \mathcal{D}\rangle$.
This will imply that $M\cong Z_2\langle i\rangle$ in $\langle\mathcal{Z}\rangle/ \langle  \mathcal{D}\rangle$.
Thus, $\overline{F}$ is essentially surjective.

It is proceed by induction on $i$.
If $i=1$ then according to the proof of (1),
we see that $Z_2[1]\cong Z_2\langle 1\rangle$ in $\langle\mathcal{Z}\rangle/ \langle  \mathcal{D}\rangle$.
Suppose that $i>1$.
Since $Z_2\langle i-1\rangle\in \mathcal{Z}$,
there exists a triangle
$$D_{Z_2\langle i-1\rangle}\longrightarrow Z_2\langle i-1\rangle\langle 1\rangle=Z_2\langle i\rangle\longrightarrow Z_2\langle i-1\rangle[1]\longrightarrow D_{Z_2\langle i-1\rangle}[1]$$
in $\langle \mathcal{Z}\rangle$ with $D_{Z_2\langle i-1\rangle}\in \mathcal{D}$ (see Theorem \ref{triangulated structure}(1)).
Therefore, $Z_2\langle i\rangle\cong Z_2\langle i-1\rangle[1]$ in $\langle\mathcal{Z}\rangle/ \langle  \mathcal{D}\rangle$.
By the induction assumption, we have $Z_2[i-1]\cong Z_2\langle i-1\rangle$ in $\langle\mathcal{Z}\rangle/ \langle  \mathcal{D}\rangle$.
Hence, $Z_2[i]\cong Z_2\langle i\rangle$ in $\langle\mathcal{Z}\rangle/ \langle  \mathcal{D}\rangle$, as desired.

(3) $\overline{F}$ is full.

Since $F=\overline{F}\pi$, it suffices to show that $F$ is full.
To this end, let $$ X\overset{f}{\longleftarrow} W\overset{g}{\longrightarrow} Y$$
be a morphism in $\langle\mathcal{Z}\rangle/ \langle  \mathcal{D}\rangle$
such that $X,Y \in \mathcal {Z}$
and $f$ lies in the compatible saturated multiplicative system corresponding to $\langle  \mathcal{D}\rangle$.
Complete $f$ to a triangle
$$X[-1]\overset{\omega}{\longrightarrow} Q\rightarrow W \overset{f}{\longrightarrow} X$$
with $Q\in \langle  \mathcal{D}\rangle$.
Since the pair $({_{\mathcal{D}}\mathcal{U}},\mathcal{U}_{\mathcal{D}})$ forms a co-t-structure on $\langle\mathcal{D}\rangle$,
there is a triangle
$$A\overset{h}{\longrightarrow} Q\overset{\varphi}{\longrightarrow} B {\longrightarrow} A[1] $$
in $\langle  \mathcal{D}\rangle$ with $A\in{_{\mathcal{D}}\mathcal{U}}[-1]$ and $B\in\mathcal{U}_{\mathcal{D}}$
(see Fact \ref{presilting CO-T}).
According to Remark \ref{properties of presilting},
we see that $\Hom_{\mathcal{T}}(X[-1],B)=0$.
This yields that $\varphi\omega=0$.
Hence, $\omega$ factors through $h$.
Consider now the following commutative of triangles
$$\xymatrix@C=20pt@R=15pt
{ X[-1] \ar[r]\ar[d]^{\|} & A \ar[r]\ar[d]^{h} & W' \ar[r]^{s}\ar[d]^{l}& X\ar[d]^{\|}\\
  X[-1] \ar[r]^{\omega}& Q \ar[r]& W \ar[r]^{f} & X}$$
where $s, l, f$ are all in the compatible saturated multiplicative system corresponding to $\langle\mathcal{D}\rangle$.
Since $\Hom_{\mathcal{T}}(A,Y)=0$ by Remark \ref{properties of presilting} again,
there exists some $k : X \to Y$ such that $gl = ks = kfl$.
So we have $k = gf^{-1}$.
Thus, $F$ is full, as desired.

(4) $\overline{F}$ is faithful.

Suppose that there exists a morphism $f:X \to Y$ in $\mathcal{Z}/[\mathcal{D}]$ such that $\overline{F}(f)=0$.
We want to show $f=0$. To this end,
complete $f$ to a triangle
$$ X\overset{f}{\longrightarrow} Y \overset{g}{\longrightarrow} Z \rightarrow X\langle1\rangle$$
in $\mathcal{Z}/[\mathcal{D}]$.
Since $\overline{F}(f)=0$, we see that $\overline{F}(g)$ is a section.
According to (3), we know that $\overline{F}$ is full.
So there exists some morphism $\alpha: Z \to Y$ such that $1_{\overline{F}(Y)} = \overline{F}(\alpha g)$.
Let $\beta = \alpha g$ and complete $\beta$ to a triangle
$$Y\overset{\beta}{\longrightarrow} Y \rightarrow C(\beta)\rightarrow Y\langle1\rangle$$
in $\mathcal{Z}/[\mathcal{D}]$.
Note that $C(\beta)$ is an object in $\mathcal{Z}$ by the construction of triangles in $\mathcal{Z}/[\mathcal{D}]$.
Since $\overline{F}(\beta)=1_{\overline{F}(Y)}$,
we conclude that $\overline{F}(C(\beta))\cong 0$ in $\langle\mathcal{Z}\rangle/ \langle  \mathcal{D}\rangle$,
i.e., $\overline{F}(C(\beta))\in\langle\mathcal{D}\rangle$.
This means $C(\beta)\in\langle\mathcal{D}\rangle$ by the definition of $\overline{F}$.
In view of Lemma \ref{key for faithful},
we see that $C(\beta)\in \mathcal{Z}\cap\langle \mathcal{D}\rangle= \mathcal{D}$.
Hence, $\beta$ is an isomorphism in $\mathcal{Z}/[\mathcal{D}]$.
This implies that $g$ is a section,
and hence, $f = 0$, as desired.
This completes the proof.
\end{proof}

\section{\bf Applications }

Various applications of Theorem \ref{thm1} will be discussed in this section.
We recover both a result of Iyama and Yang and a result of the third author (see Corollary \ref{wei's} and Corollary \ref{Yang1}).
We extend the classical Buchweitz's triangle equivalence
from Iwanaga-Gorenstein rings to Noetherian rings
(see Corollary \ref{Gproj} and Remark \ref{Buchweitz}).
We obtain the converse of Buchweitz's triangle equivalence and a result of Beligiannis (see Corollary \ref{inverse1} and Corollary \ref{inverse2}),
and give characterizations for Iwanaga-Gorenstein rings and Gorenstein algebras (see Corollary \ref{commutative N}).

\subsection{Wei's triangle equivalence}\label{wei}

Suppose that $\mathcal{M}$ is a presilting subcategory of $\mathcal{T}$.
According to \cite[Proposition 2.5]{W1},
we see that $({_\mathcal{M}\mathscr{X}} \cap {\mathscr{X}_\mathcal{M}},{_\mathcal{M}\mathscr{X}} \cap {\mathscr{X}_\mathcal{M}})$
forms a $\mathcal{M}$-mutation pair.
Hence, as a consequence of Theorem \ref{thm1},
we get the following result.

\begin{cor}\label{triangle equivalence1}
Let $\mathcal{M}$ be a presilting subcategory of $\mathcal{T}$.
Then there exists a triangle equivalence
$${_\mathcal{M}\mathscr{X}} \cap {\mathscr{X}_\mathcal{M}}\,/\,[\mathcal{M}]\,\,\simeq\,\,
\langle  {_\mathcal{M}\mathscr{X}} \cap {\mathscr{X}_\mathcal{M}}\rangle/ \langle \mathcal{M}\rangle.$$
\end{cor}

\begin{lem}\label{key wei}
Let $\mathcal{M}$ be a presilting subcategory of $\mathcal{T}$.
Then we have
$$\langle  {_\mathcal{M}\mathscr{X}} \cap {\mathscr{X}_\mathcal{M}}\rangle\,\,=\,\,
\langle  {_\mathcal{M}\mathscr{X}}\rangle \cap \langle{\mathscr{X}_\mathcal{M}}\rangle.$$
\end{lem}

\begin{proof}
Obviously, we have $\langle  {_\mathcal{M}\mathscr{X}} \cap {\mathscr{X}_\mathcal{M}}\rangle\subseteq
\langle  {_\mathcal{M}\mathscr{X}}\rangle \cap \langle{\mathscr{X}_\mathcal{M}}\rangle$.
Hence, it remains to show that the converse containment holds.

Let $M$ be an object in $\langle{_\mathcal{M}\mathscr{X}}\rangle \cap \langle{\mathscr{X}_\mathcal{M}}\rangle$.
Then by Lemma \ref{necessary facts}(1) and Lemma \ref{thick}(1),
we see that $M\in\widehat{{\mathscr{X}_\mathcal{M}}}$.
Note that $\mathcal{M}$ is clear a weak-cogenerator in $\mathscr{X}_\mathcal{M}$.
It follows from Theorem \ref{A-B Approximation1} that
there exists a triangle
$$ M\to K^M\to X^M\to M[1]$$
in $\mathcal{T}$ with $X^M\in\mathscr{X}_\mathcal{M}$ and $K^M\in\widehat{\mathcal{ M}}\subseteq \langle\mathcal{ M}\rangle\subseteq\langle{_\mathcal{M}\mathscr{X}}\rangle$.
Since $M$ belongs to $\langle{_\mathcal{M}\mathscr{X}}\rangle$ by assumption,
we deduce that $X^M \in \langle{_\mathcal{M}\mathscr{X}}\rangle=\widecheck{{_\mathcal{M}\mathscr{X}}}$ by Lemma \ref{necessary facts}(2) and Lemma \ref{thick}(2).
Note that $\mathcal{M}$ is also a weak-generator in ${_\mathcal{M}\mathscr{X}}$.
By Theorem \ref{A-B Approximation2}, we have a triangle
$$X^M\to  {_MX}\to {_ML}\to X^M[1]$$
in $\mathcal{T}$ with $_MX\in{_\mathcal{M}\mathscr{X}}$ and
$_ML\in \widecheck{\mathcal{M}}\subseteq \langle\mathcal{ M}\rangle\subseteq \langle  {_\mathcal{M}\mathscr{X}} \cap {\mathscr{X}_\mathcal{M}}\rangle$.
Since both $X^M$ and ${_ML}$ belong to $\mathscr{X}_\mathcal{M}$
and $\mathscr{X}_\mathcal{M}$ is closed under extensions (see Lemma \ref{necessary facts}(1)),
we get that ${_MX}\in{_\mathcal{M}\mathscr{X}} \cap {\mathscr{X}_\mathcal{M}}\subseteq \langle  {_\mathcal{M}\mathscr{X}} \cap {\mathscr{X}_\mathcal{M}}\rangle$.
This implies that $X^M\in\langle  {_\mathcal{M}\mathscr{X}} \cap {\mathscr{X}_\mathcal{M}}\rangle$ as well.
Note that $K^M\in\langle\mathcal{ M}\rangle\subseteq \langle  {_\mathcal{M}\mathscr{X}} \cap {\mathscr{X}_\mathcal{M}}\rangle$.
It follows that $M\in\langle  {_\mathcal{M}\mathscr{X}} \cap {\mathscr{X}_\mathcal{M}}\rangle$.
Hence, $\langle{_\mathcal{M}\mathscr{X}}\rangle \cap \langle{\mathscr{X}_\mathcal{M}}\rangle\subseteq\langle  {_\mathcal{M}\mathscr{X}} \cap {\mathscr{X}_\mathcal{M}}\rangle$, as desired.
\end{proof}

As a consequence of Corollary \ref{triangle equivalence1} and Lemma \ref{key wei},
we obtain the following triangle equivalence,
which is the main result in \cite{W1}.

\begin{cor}\cite[Theorem A]{W1}\label{wei's}
Let $\mathcal{M}$ be a presilting subcategory of $\mathcal{T}$.
Then there exists a triangle equivalence
$${_\mathcal{M}\mathscr{X}} \cap {\mathscr{X}_\mathcal{M}}\,/\,[\mathcal{M}]\,\,\simeq\,\,
\langle  {_\mathcal{M}\mathscr{X}}\rangle \cap \langle{\mathscr{X}_\mathcal{M}}\rangle/ \langle \mathcal{M}\rangle.$$
\end{cor}

In 2012, Aihara and Iyama introduced in \cite{AI} the notion of a \emph{silting reduction} of $\mathcal{T}$
as the triangulated quotient $\mathcal{T}/\langle\mathcal{M}\rangle$,
where $\mathcal{M}$ is a presilting subcategory of $\mathcal{T}$.
Recently, Iyama and Yang proved the following result
which shows that under some conditions such a silting reduction can be realized as a certain subfactor triangulated category of $\mathcal{T}$.
According to \cite[Corollary 2.7]{W1},
we see that the following Iyama and Yang's result can be deduced from \cite[Theorem A]{W1} (see Corollary \ref{wei's}),
and, hence, is also a special case of our realization in Theorem \ref{thm1}.

\begin{cor}\cite[Theorem 3.6]{IYa}\label{Yang1}
Let $\mathcal{M}$ be a presilting subcategory of $\mathcal{T}$.
Suppose that

$(1)$ $\mathcal{M}$ is functorially finite in $\mathcal{T}$, and

$(2)$ for any object $X\in\mathcal{T}$,
$\Hom_{\mathcal{T}}(X, \mathcal{M}[i]) =0 =\Hom_{\mathcal{M}}(\mathcal{M}, X[i])$ for $i \gg 0$.\\
Then there exists a triangle equivalence
$${^{\perp_{i>0}}\mathcal{M}}\cap\mathcal{M}^{\perp_{i>0}}\,/\,[\mathcal{M}]\,\,\simeq\,\,\mathcal{T}/\langle\mathcal{M}\rangle.$$
\end{cor}

\subsection{\bf Applications for projective Frobenius subcategories}\label{PFS}

Throughout this subsection, let $\mathcal{A}$ be an abelian category with enough projective objects.
The subcategory of $\mathcal{A}$ consisting of all projective objects is denoted by $\mathcal{P}$.

Let $\mathcal{G}$ be a subcategory of $\mathcal{A}$ closed under extensions.
Then $\mathcal{G}$ becomes an exact category whose elements in the exact structure
are just short exact sequences in $\mathcal{A}$ such that all terms belong to $\mathcal{G}$.
We say that $\mathcal{G}$ is a \emph{Frobenius subcategory} of $\mathcal{A}$
if $\mathcal{G}$ forms a Frobenius category with respect to such an exact structure.

The following observation plays a key role to connect our main result Theorem \ref{thm1} with applications on Frobenius subcategories.

\begin{lem}\label{mutation}
Let $\mathcal{W}\subseteq\mathcal{G}$ be two subcategories of $\mathcal{A}$.
Then $\mathcal{G}$ is a Frobenius subcategory of $\mathcal{A}$
whose projective-injective objects are precisely
the objects in $\mathcal{W}$
if and only if the pair $(\mathcal{G,G})$ of subcategories $($considered in $D^\mathrm{b}(\mathcal{A}))$
forms a $\mathcal{W}$-mutation pair in $D^\mathrm{b}(\mathcal{A})$.
\end{lem}

\begin{proof}
It is an easy observation that $\mathcal{G}$ is closed under extensions in $\mathcal{A}$
if and only if it is closed under extensions in $D^\mathrm{b}(\mathcal{A})$.
Moreover, since $\mathcal{A}$ has enough projective objects by assumption,
there exists an isomorphism $\Hom_{D^\mathrm{b}(\mathcal{A})}(M,N[1])\cong \Ext_{\mathcal{A}}^1(M,N)$
for all objects $M,N$ in $\mathcal{A}$.

For the `only-if' part, the condition (2) of Definition \ref{Df of Mutation} can be obtained from the property
that $\mathcal{G}$ has enough projectives and injectives.
The condition (3) of Definition \ref{Df of Mutation} can be guaranteed by the above isomorphism and
the fact that the objects in $\mathcal{W}$ are both projective and injective in $\mathcal{G}$.

For the `if' part, note that $(\mathcal{G,G})$ forms a $\mathcal{W}$-mutation pair in $D^\mathrm{b}(\mathcal{A})$ by assumption.
The condition (3) in Definition \ref{Df of Mutation} together with the above isomorphism
show that the objects in $\mathcal{W}$ are both projective and injective in $\mathcal{G}$.
The condition (2) in Definition \ref{Df of Mutation} implies that
$\mathcal{G}$ has enough projectives and injectives.
Finally, it is routine to check that
in the exact category $\mathcal{G}$ the subcategory of projective objects coincides with the subcategory of injective objects,
and projective-injective objects are objects in $\mathcal{W}$.
\end{proof}

In order to give more applications of Theorem \ref{thm1},
it is convenient to introduce the following notion of a projective Frobenius subcategory of $\mathcal{A}$.
Recall that a subcategory $\mathcal{G}$ of $\mathcal{A}$ is called \emph{resolving} if
it contains $\mathcal{P}$, and is closed under extensions and kernels of epimorphisms.

\begin{df}\label{Projective Frobenius}
{\rm
A Frobenius subcategory $\mathcal{G}$ of $\mathcal{A}$ is said to be \emph{projective}
provided that

(1) It is resolving.

(2) Its projective-injective objects are just the projective objects of $\mathcal{A}$.

As usual, we denote the stable category $\mathcal{G}/[\mathcal{P}]$ by $\underline{\mathcal{G}}$ ;
it is a triangulated category (we refer the reader to \cite{Ha} for more details).
}
\end{df}

As a consequence of Theorem \ref{thm1} and Lemma \ref{mutation},
we obtain

\begin{cor}\label{equivalence}
Let $\mathcal{G}$ be a projective Frobenius subcategory of $\mathcal{A}$.
Then there exists a triangle equivalence
$$\underline{\mathcal {G}}\,\,\simeq\,\, \langle\mathcal{G}\rangle/ \langle\mathcal{P}\rangle.$$
\end{cor}

Inspired by the Christensen's notion of Gorenstein projective dimension for complexes in $D^-(\mathcal{A})$ (see \cite[Subsection 1.7]{lwc2}),
we introduce the following notion of projective dimension with respect to a projective Frobenius subcategory of $\mathcal{A}$.

\begin{df}\label{dimension of complex}{\rm
Let $\mathcal{G}$ be a projective Frobenius subcategory of $\mathcal{A}$
and $X$ a complex in $D^-(\mathcal{A})$.
The $\mathcal{G}$-\emph{projective} \emph{dimension} of $X$, denoted by $\mathcal{G}$-$\dim X$, is defined as

\begin{center}
$\begin{aligned}\mathcal{G}\textrm{-}\dim X=\inf\,\{\,\sup\,\{\,l\in\mathbb{Z}\mid G_{-l}\neq0\,\}\mid\, \,& G\simeq X,
\textrm{~where}~G~\textrm{is a bounded-above complex}\\ & \textrm{with each~} G_{-l}\in\mathcal{G}\,\}.
\end{aligned}$
\end{center}

}\end{df}

Following from the definition above,
we have the next assertions.

\begin{rem}\label{closed under equi}
{\rm Suppose that $\mathcal{G}$ is a projective Frobenius subcategory of $\mathcal{A}$.

$(1)$ For any complex $X\in D^-(\mathcal{A})$ and any $k\in\mathbb{Z}$,
      we see that $\mathcal{G} \textrm{-}\dim(X[k])=\mathcal{G} \textrm{-}\dim X+k$.

$(2)$  Note that the $\mathcal{G}$-projective dimension of complexes is defined based upon the equivalence relation of complexes.
       This implies that the subcategory of $D^-(\mathcal{A})$
       consisting of all complexes with finite $\mathcal{G}$-projective dimension
       is closed under equivalences of complexes.

$(3)$  Let $M$ be an object in $\mathcal{A}$.
       Then the $\mathcal{G}$-projective dimension of $M$ (considered as a stalk complex)
       is the least non-negative integer $n$ such that there exists an exact sequence
       $$0 \to G_n \to \cdots \to G_1 \to G_0 \to M \to 0$$
       with $G_i \in\mathcal{ G}$ for all $0\leqslant i\leqslant n$.
}\end{rem}

One can obtain the coming result by a similar argument to the proof of \cite[Theorem 3.9(1)]{V}.

\begin{lem}\label{second in third}
Let $\mathcal{G}$ be a projective Frobenius subcategory of $\mathcal{A}$.
Suppose that $$ M \to M' \to M''\to M[1]$$ is a triangle in $D^\mathrm{b}(\mathcal{A})$.
If any two complexes of $M$, $M'$ and $M''$ have finite $\mathcal{G}$-projective dimension,
then so does the third.
\end{lem}

Suppose that $\mathcal{G}$ is a projective Frobenius subcategory of $\mathcal{A}$.
In what follows, denote by $D^\mathrm{b}(\mathcal{A})_{\widehat{\mathcal{G}}}$ the subcategory of $D^\mathrm{b}(\mathcal{A})$
consisting of all complexes with finite $\mathcal{G}$-projective dimension.
Combining Remark \ref{closed under equi} with Lemma \ref{second in third},
one can conclude that $D^\mathrm{b}(\mathcal{A})_{\widehat{\mathcal{G}}}$ is a triangulated subcategory of $D^\mathrm{b}(\mathcal{A})$.

\begin{cor}\label{key cor}
Let $\mathcal{G}$ be a projective Frobenius subcategory of $\mathcal{A}$.
Then there exists a triangle equivalence
$$\underline{\mathcal {G}}\,\,\simeq\,\, D^\mathrm{b}(\mathcal{A})_{\widehat{\mathcal{G}}}/ K^\mathrm{b} (\mathcal{P}).$$
\end{cor}

\begin{proof}
Obviously, we have
$\langle  \mathcal{P}\rangle=K^\mathrm{b} (\mathcal{P})$
and $\langle\mathcal{G}\rangle=K^\mathrm{b} (\mathcal{G})/K^\mathrm{b}_{ac} (\mathcal{G})$,
where $K^\mathrm{b}_{ac} (\mathcal{G})$ is the subcategory of $K^\mathrm{b} (\mathcal{G})$
consisting of all acyclic complexes.
Hence, to complete the proof, by Corollary \ref{equivalence},
we need only to show that $K^\mathrm{b} (\mathcal{G})/K^\mathrm{b}_{ac} (\mathcal{G})\cong D^\mathrm{b}(\mathcal{A})_{\widehat{\mathcal{G}}}$.
Indeed, it is clear by the definition of $\mathcal{G}$-projective dimension for homology bounded complexes.
\end{proof}

Recall that an object $M$ in $\mathcal{A}$ is called \emph{Gorenstein projective} \cite{EJ2000}
if there exists an exact complex $P$ of projective objects
such that $M$ is isomorphic to a cokernel of $P$
and the complex $\Hom_{\mathcal{A}}(P,Q)$ is still exact whenever $Q$ is a projective object.
Denote by $\mathcal {GP}$ the subcategory of all Gorenstein projective objects in $\mathcal{A}$.

Clearly, the subcategory $\mathcal {GP}$ is a projective Frobenius subcategory of $\mathcal{A}$.
Hence, by Corollary \ref{key cor}, we have

\begin{cor}\label{Gproj}
Let $\mathcal{A}$ be an abelian category with enough projective objects.
Then there exists a triangle equivalence
$$\underline{\mathcal {GP}}\,\,\simeq\,\, D^\mathrm{b}(\mathcal{A})_{\widehat{\mathcal{GP}}}/ K^\mathrm{b} (\mathcal{P}).$$
\end{cor}

\begin{rem}\label{Buchweitz}
{\rm
(1) In 1986, Buchweitz \cite{Buch} established the following famous triangle equivalence over an \emph{Iwanaga-Gorenstein} ring $R$
(that is, $R$ is a left and right Noetherian ring with finite self-injective dimension on both sides (see, e.g., \cite{EJ2000}))
$$\quad\quad\quad\underline{\mathcal{G}p}\,\,\simeq\,\,D^\mathrm{b}(\textrm{mod}\,R)/ K^\mathrm{b}(\textrm{proj}\,R)\quad\quad\quad(\natural),$$
where $\mathcal{G}p$ denotes the subcategory of mod$\,R$
consisting of all finitely generated Gorenstein projective right $R$-modules.

Assume now that $R$ is a right Noetherian ring.
Then mod$\,R$ is an abelian category with enough projective objects.
Thus, by Corollary \ref{Gproj},
we obtain in this case the triangle equivalence
$$\underline{\mathcal{G}p}\,\,\simeq\,\,D^\mathrm{b}(\textrm{mod}\,R)_{\widehat{\mathcal{G}p}}/ K^\mathrm{b}(\textrm{proj}\,R),$$
which extends $(\natural)$ from Iwanaga-Gorenstein rings to arbitrary Noetherian rings.

(2) The `big' version of the triangle equivalence $(\natural)$ was proved by Beligiannis \cite{BE1}
over a ring $R$ which has finite right Gorenstein global dimension
(that is, a ring satisfies that the supremum of Gorenstein projective dimension of all right $R$-modules is finite).
More specifically, over such a ring $R$, Beligiannis showed that
there exists a triangle equivalence
$$\quad\quad\quad\underline{\mathcal{G}P}\,\,\simeq\,\,D^\mathrm{b}(\M\,R)/ K^\mathrm{b}(\textrm{Proj}\,R)\quad\quad\quad(\sharp)$$
in which $\mathcal{G}P$ denotes the subcategory of $\M\,R$ consisting of all Gorenstein projective right $R$-modules
(indeed, in Beligiannis' work \cite{BE1},
he established $(\sharp)$ over rings satisfying that any projective right module has finite injective dimension
and any injective right module has finite projective dimension.
He called therein such rings \emph{right Gorenstein rings}.
However, in view of \cite[Theorem 6.9]{BE1} and \cite[Theorem 4.1]{E4},
we see that right Gorenstein rings are just rings having finite right Gorenstein global dimension).

Let now $R$ be an arbitrary ring.
Then as a consequence of Corollary \ref{Gproj},
we see that the following triangle equivalence holds true
$$\underline{\mathcal{G}P}\,\,\simeq\,\,D^\mathrm{b}(\M\,R)_{\widehat{\mathcal{G}P}}/ K^\mathrm{b}(\textrm{Proj}\,R).$$
It generalizes $(\sharp)$ from rings with finite Gorenstein global dimension to arbitrary rings.
}\end{rem}

\subsection{\bf The converse of Buchweitz's triangle equivalence}

As usual, denote by $\G_\mathcal{A}\,M$ the Gorenstein projective dimension of an object $M\in\mathcal{A}$.
When we consider a right $R$-module $M$ in $\M\,R$ or mod$\,R$,
the Gorenstein projective dimension of $M$ will be denoted simply by $\G_R\,M$.

\begin{lem}\label{lem4.10}
Let $\mathcal{A}$ be an abelian category with enough projective objects.
Then every object in $\mathcal{A}$ has finite Gorenstein projective dimension if and only if
$$D^{\rm b}(\mathcal{A})\,\,=\,\,D^{\rm b}(\mathcal{A})_{\widehat{\mathcal {G}\mathcal {P}}}.$$
\end{lem}
\begin{proof}
The if-part is obvious.

For the only-if-part,
we need to show that any homology bounded complex $M$
has finite Gorenstein projective dimension.
To this end, let $P$ be a bounded-above complex of projective objects such that $P\cong M$ in $D^{\rm b}(\mathcal{A})$,
and suppose that $\inf\{i\in\mathbb{Z}\,|\,\h_{i}(M)\neq0\} =s$ for some integer $s$.
Then $P$ is exact in degrees $i< s$.
This implies that the following sequence
$$\cdots\to P_{s-1} \to P_{s}\to \C_s(P)\to 0\to\cdots\indent\indent(\ddag)$$
is exact.
Note that $(\ddag)$ is indeed a projective resolution of $\C_s(P)$
and $\C_s(P)$ has finite Gorenstein projective dimension by assumption.
It follows from \cite[Theorem 2.11]{Holm4} that $\C_{s-n}(P)$ is Gorenstein projective,
where $n=\G_\mathcal{A}\,\C_{s}(P)$.
Thus, we conclude that $M$ has finite Gorenstein projective dimension,
as desired.
\end{proof}

As a consequence of Corollary \ref{Gproj} and Lemma \ref{lem4.10},
we have the next result.

\begin{cor}\label{cor4.9}
Let $\mathcal{A}$ be an abelian category with enough projective objects.
Then every object in $\mathcal{A}$ has finite Gorenstein projective dimension if and only if
there exists a triangle equivalence
$$\underline{\mathcal {G}\mathcal {P}}\,\, \simeq\,\,
  D^{\rm b}(\mathcal{A})/ K^{\rm b}(\mathcal {P}).$$
\end{cor}

It is well-known that mod$\,R$ is an abelian category if and only if $R$ is a right Noetherian ring,
and that in this case mod$\,R$ has enough projective objects.
Hence, by Corollary \ref{cor4.9}, we obtain the coming result.
It gives the converse of the Buchweitz's triangle equivalence (see Remark \ref{Buchweitz}(1)).

\begin{cor}\label{inverse1}
Let $R$ be a right Noetherian ring.
Then the following statements are equivalent:

$(1)$ Every finitely generated right $R$-module has finite Gorenstein projective dimension.

$(2)$ There exists a triangle equivalence
      $\underline{\mathcal {G}p}\simeq D^\mathrm{b}(\mathrm{mod}\,R)/ K^\mathrm{b}(\mathrm{proj}\,R).$
\end{cor}

The next result have been obtained by Beligiannis in \cite[Theorem 6.9]{BE1}.
We can also obtain it by Corollary \ref{cor4.9} together with
the fact that a ring $R$ has finite right Gorenstein global dimension
if and only if every right $R$-module has finite Gorenstein projective dimension
(see \cite[Theorem 4.1]{E4}).

\begin{cor}$($\cite[Theorem 6.9]{BE1}$)$\label{inverse2}
Let $R$ be a ring.
Then the following statements are equivalent:

$(1)$ $R$ has finite right Gorenstein global dimension.

$(2)$ There exists a triangle equivalence
      $\underline{\mathcal {G}P} \simeq D^\mathrm{b}(\M\,R)/ K^\mathrm{b}(\mathrm{Proj}\,R).$
\end{cor}

Based on Corollary \ref{inverse1} and Corollary \ref{inverse2},
we obtain the following result,
which can characterize Iwanaga-Gorenstein rings and Gorenstein algebras.
Note that the equivalence of (1) and (3) is a special case of Bergh, J{\o}rgensen and Oppermann's result \cite[Theorem 3.6]{BJS},
where it is proved that if $R$ is either a left and right Artin ring or
a commutative Noetherian local ring,
then there exists a triangle equivalence
$\underline{\mathcal {G}p} \simeq D^\mathrm{b}(\mathrm{mod}\,R)/ K^\mathrm{b}(\mathrm{proj}\,R)$
if and only if $R$ is Gorenstein.

\begin{cor}\label{commutative N}
Let $R$ be a left and right Noetherian ring.
Then the following statements are equivalent:

$(1)$ $R$ is Iwanaga-Gorenstein.

$(2)$ There exists a triangle equivalence
      $\underline{\mathcal {G}P} \simeq D^\mathrm{b}(\M\,R)/ K^\mathrm{b}(\mathrm{Proj}\,R).$\\
If $R$ is further an Artin algebra,
then the above two conditions are equivalent to

$(3)$  There exists a triangle equivalence
      $\underline{\mathcal {G}p} \simeq D^\mathrm{b}(\mathrm{mod}\,R)/ K^\mathrm{b}(\mathrm{proj}\,R).$
\end{cor}

\begin{proof}
According to \cite[Theorem 12.3.1]{EJ2000},
we know that a left and right Noetherian ring is Iwanaga-Gorenstein
if and only if it has finite right Gorenstein global dimension.
Hence, in view of Corollary \ref{inverse2},
we deduce that the statements (1) and (2) are equivalent.

Suppose that $R$ is further an Artin algebra.
Then the category $\mathrm{mod}\,R$ is an abelian category with enough projectives and injectives.
It follows from \cite[Theorem 4.16]{BE1} that
if every finitely generated right $R$-module has finite Gorenstein projective dimension,
then their Gorenstein projective dimension has a supremum.
Hence, by virtue of \cite[Theorem 12.3.1]{EJ2000} again,
we see that $R$ is a Gorenstein algebra if and only if every finitely generated right $R$-module has finite Gorenstein projective dimension.
Thus, by Corollary \ref{inverse1}, we conclude that the equivalence of (1) and (3) holds true.
This completes the proof.
\end{proof}

\centerline {\bf Acknowledgments}
The authors would like to thank the anonymous referee for many helpful comments and suggestions.
The authors are indebted to Xiaoxiang Zhang, Xiaowu Chen, Yu Ye and Junpeng Wang for their helpful discussions.
This work is supported by the National Natural Science Foundation of China (Grant nos. 11601433, 11261050 and 11771212),
the Postdoctoral Science Foundation of China (Grant no. 2106M602945XB)
and Northwest Normal University (Grant no. NWNU-LKQN-15-12).

\end{document}